\documentclass[11pt]{amsart}
\usepackage{amsthm}
\usepackage{array}
\usepackage{amsmath}
\usepackage{enumerate}
\usepackage{tikz}
\usetikzlibrary{calc}
\usepackage{amssymb}
\usepackage{latexsym}
\usepackage{amsfonts}
\usepackage{color}
\usepackage{mathrsfs}
\usepackage{epsfig,helvet}

\theoremstyle{plain} \numberwithin{equation}{section}
\newtheorem{thm}{Theorem}[section]
\newtheorem{theorem}[thm]{Theorem}
\newtheorem{lemma}[thm]{Lemma}
\newtheorem{corollary}[thm]{Corollary}
\newtheorem{example}[thm]{Example}
\newtheorem{definition}[thm]{Definition}

\begin{document}
	\setcounter{page}{1}

	\title[Isoclinism and factor set in regular Hom-Lie Superalgebras ]{  Isoclinism and factor set in regular Hom-Lie Superalgebras }

	\author[NANDI]{N. Nandi}
	\address{Department of Mathematics, National Institute of Technology  \\
		Rourkela, 
		Odisha-769028 \\
		India}
	\email{nupurnandi999@gmail.com}
	\author[Padhan]{R. N. Padhan}
	\address{Centre for Applied Mathematics and Computing, Institute of Technical Education and Research  \\
		Siksha `O' Anusandhan (A Deemed to be University)\\
		Bhubaneswar-751030 \\
		Odisha, India}
	\email{rudra.padhan6@gmail.com, rudranarayanpadhan@soa.ac.in}
	
	\author[Pati]{K. C. Pati}
	\address{Department of Mathematics, National Institute of Technology  \\
		Rourkela, 
		Odisha-769028 \\
		India}
	\email{kcpati@nitrkl.ac.in}
	
	\maketitle

\begin{abstract}
	Hom-Lie superalgebras can be considered as the deformation of Lie superalgebras; which are $\mathbb{Z}_2$-graded generalization of Hom-Lie algebras. The motivation of this paper is to introduce the concept of isoclinism and factor set in regular Hom-Lie superalgebras. Finally, we obtain that two finite same dimensional regular Hom-Lie superalgebras are isoclinic if and only if they are isomorphic.
\end{abstract}

\section{Introduction}
Hartwig et al. studied Hom-Lie algebras as a part of a study of deformations of the Witt and the Virasoro algebras \cite{hartwig2006}. Lie superalgebra was introduced by Kac which is the $\mathbb{Z}_2$-graded Lie algebra \cite{kac1977}. Then Hom-Lie algebra was generalized to Hom-Lie superalgebra by Ammar et al. \cite{ammar2010}. Cohomology of Hom-Lie superalgebras and $q$-deformed Witt superalgebras was studied in \cite{ammar2013}.
\par
As is well known, isoclinism plays an vital role in classification of finite $p$-group. The notion of isoclinism for group was introduced by Hall in 1940 which is weaker than isomorphism. In 1993, Moneyhun used this concept on Lie algebras see \cite{moneyhun1994,moneyhunK1994}. Recently, Nayak \cite{nayak2018} and Padhan et al. \cite{padhan2019} studied isoclinism for Lie superalgebras. Isocinism for $n$-Lie superalgebras is studied in \cite{khuntia2022,padhan2022}. Factor set in Lie algebras was defined by Moneyhun \cite{moneyhun1994}, and the same for pair of Lie algebras was defined by Moghaddam et al. in \cite{eshrati2016}. It is also defined for Lie superalgebras by Nayak et al. \cite{nayak2019}. 
\par In this paper, we define isoclinism for regular Hom-Lie superalgebras which is not true in general for any arbitary Hom-Lie superalgebras. Furthermore the factor set in regular Hom-Lie superalgebras is defined and some of its properties are discussed. We generalize some results for regular Hom-Lie algebras \cite{padhan2019} to regular Hom-Lie superalgebras. One can refer to \cite{sheng2013,agrebaoui2019} for the other exploring works of this field.

\par In this paper all vector superspaces are considered over $\mathbb{F}$, a field of characteristic $0$ and $\mathbb{Z}_2=\{\overline{0},\overline{1}\}$ the additive group of two elements. A superspace is a $\mathbb{Z}_2$-graded vector space $V={V}_{\overline{0}}\oplus {V}_{\overline{1}}$. A $sub superspace$ is a $\mathbb{Z}_2$-graded vector space which is closed under bracket operation. For a homogeneous element $m\in V_{\overline{0}}\cup V_{\overline{1}}$, we write $|m|$ for the parity $\gamma\in \mathbb{Z}_2$. For any superalgebra in this paper, homomorphisms always mean even homomorphisms.
\par  Let us begin with some definitions related to Hom-Lie superalgebras.
\begin{definition}\label{d1}  
	A Hom-Lie superalgebra is a triple $(G,[.,.],\theta)$ which equipped with a $\mathbb{Z}_2$-graded vector space $G$, a bilinear map $[.,.]:G\times G\rightarrow G$, and a homomorphism $\theta:G\rightarrow G$ satisfying;
	\begin{enumerate}
		\item $[n_1,n_2]=-(-1)^{|n_1||n_2|}[n_2,n_1]$ (graded skew-symmetry),
		\item  $(-1)^{|n_1||n_3|}[\theta(n_1),[n_2,n_3]]+(-1)^{|n_3||n_2|}[\theta(n_3),[n_1,n_2]]+(-1)^
		{|n_2||n_1|}[\theta(n_2),[n_3,n_1]]=0$ (graded Hom-Jacobi identity),
	\end{enumerate}
	for homogeneous elements $n_1,n_2,n_3\in G.$ 
\end{definition}

Suppose $(G,[.,.],\theta)$ is a Hom-Lie superalgebra. A Hom-Lie subspace $W\subseteq G $ is a $Hom$-$subalgebra$ of $G$ if $\alpha(W) \subseteq W$ and $W$ is closed under the bracket operation $[.,.]$, i.e., $[W,W]\subseteq W$. A Hom-subalgebra $W$ is called a $graded ~ideal$ of $G$ if $[W,G]\subseteq W$. The $center$ of the Hom-Lie superalgebra $(G,[.,.],\theta)$, denoted by $Z(G)$ is defined by $Z(G)=\{x\in G| [x,y]=0$ {\rm{for}} $y\in G\}.$ An $abelian$ Hom-Lie superalgebra is a vector superspace $G$ equipped with trivial bracket and an even linear map $\theta:G\rightarrow G$. A Hom-Lie superalgebra $(G,[.,.],\theta)$ is called $multiplicative$ if $\theta([n_1,n_2])=[\theta(n_1),\theta(n_2)]$ for $n_1,n_2\in G$. A multiplicative Hom-Lie superalgebra $(G,[.,.],\theta)$ is said to be $regular$ if $\theta$ is bijective.

In this paper we only consider multiplicative Hom Lie superalgebras.
\begin{example}
	Taking $\theta=Id$ in Definition \ref{d1}, we obtain the definition of a Lie superalgebra.
\end{example}

\begin{lemma}\label{l3}
	If $(G,[.,.],\theta)$ is a regular Hom-Lie superalgebra, then $Z(G)$ is a Hom-ideal of $G$.
\end{lemma}
\begin{proof}
	Let $z\in Z(G)$. Suppose $m$ and $u$ are two homogeneous elements of $G$, we put $m=\theta(u)$, then we have \[ [\theta(z),m]= [\theta(z),\theta(u)]=\theta([z,u])=0, \]
	which implies $\theta(Z(G))\subseteq  Z(G)$.
	Let $z_1,z_2 \in Z(G)$, then 
	\[[[z_1,z_2],m]=[[z_1,z_2],\alpha(u)]=0,\]
	which means that $Z(G)$ is a Hom-ideal of $G$.
\end{proof}
The concept of stem Lie superalgebra is defined and studied in [6].
\begin{definition}\label{d4}
	If $Z(G)\subseteq G^{'}$ then a regular Hom-Lie superalgebra $(G,[.,.],\alpha)$ is called stem Hom-Lie superalgebra.  
\end{definition}

\begin{definition}\label{d5}
	Let $(G_1,[.,.]_1,\theta_1)$ and $(G_2,[.,.]_2,\theta_2)$ be two Hom-Lie superalgebras. A homomorphism from $f:(G_1,[.,.]_1,\theta_{1})\rightarrow (G_2,[.,.]_2,\theta_{2})$ is an even linear map $f:G_1\rightarrow G_2$ satisfying $f([n_1,n_2]_1)=[f(n_1),f(n_2)]_2$ and $f\theta_1=\theta_2f$ for $n_1,n_2\in G_1$. In particular, the following diagram is commutative:
	\begin{center}
		\begin{tikzpicture}[>=latex]
			\node (x) at (0,0) {\(G_1\)};
			\node (z) at (0,-2) {\(G_1\)};
			\node (y) at (2,0) {\(G_2\)};
			\node (w) at (2,-2) {\(G_2.\)};
			\draw[->] (x) -- (y) node[midway,above] {$f$};
			\draw[->] (x) -- (z) node[midway,left] {$\theta_1$};
			\draw[->] (z) -- (w) node[midway,below] {$f$};
			\draw[->] (y) -- (w) node[midway,right] {$\theta_2$};
		\end{tikzpicture}\\
	\end{center}
	They are isomorphic if $f:G_1\rightarrow G_2$ is bijective .	
\end{definition}

\begin{lemma}\label{l6}
	Suppose $f:(G_1,[.,.]_1,\theta _1)\longrightarrow (G_2,[.,.]_2,\theta_2)$ is an isomorphism. If $(G_1,[.,.]_1,\theta _1)$ is regular then $(G_2,[.,.]_2,\theta _2)$ is also regular.
\end{lemma}
\begin{proof}
	Suppose $n_1,n_2 \in G_2$ then there exists $m_1,m_2 \in G_1$ such that $f(m_1)=n_1$ and $f(m_2)=n_2$. As $f \theta_1=\theta_2 f$ and $\theta_1$ is regular, we have
	\begin{align*}
		\theta_{2}([n_1,n_2]_2)&=\theta_{2}([f(n_1),f(n_2)]_2)=\theta_{2} f([n_1,n_2]_1)\\
		&=f \theta_{1}([n_1,n_2]_1)=f([\theta_1(n_1),\theta_1(n_2)]_1)\\
		&=[f\theta_1(n_1),f\theta_1(n_2)]_2
		=[\theta_2 f (n_1), \theta_2 f (n_2)]_2=[\theta_2(n_1),\theta_2(n_2)]_2.
	\end{align*}
	\noindent Now $f$ and $\theta_1$ are bijective, so $\theta_2$. Hence $(G_2,[.,.],\theta _2)$ is regular. 
\end{proof}

\noindent  $(G/K,[.,.],\theta)$ is a Hom-Lie superalgebra with skew-bilinear and linear map and known as the quotient Hom-Lie superalgebra. For any Hom-ideal $K$ of $(G,[.,.],\theta)$, we can define quotient Hom-Lie superalgebra on the quotient vector superspace $G/K$ by defining $[.,.]:G/K\times G/K\rightarrow G/K$ by $[\overline{n_{1}},\overline{n_{2}}]=\overline{[n_1,n_2]}$ for $\overline{n_1},\overline{n_2}\in G/K$ and $\tilde{\theta}:G/K\rightarrow G/K$ is induced by $\theta$, i.e., $\tilde{\theta}(\overline{m})=\theta(m)+K.$
\par Suppose that $(G_1,[.,.]_{G_1},\theta_1)$ and $(G_2,[.,.]_{G_2},\theta_2)$ are two Hom-Lie superalgebras, define the direct sum of these Hom-Lie superalgebras $(G_1\oplus G_2,[.,.]_{G_1\oplus G_2},\Gamma)$ as: $$[(m_1,n_1),(m_2,n_2)]_{G_1\oplus G_2}=([m_1,m_2]_{G_1},[n_1,n_2]_{G_2}),$$ and the linear map $\Gamma:G_1\oplus G_2\rightarrow G_1\oplus G_2 $ is given by $$\Gamma(m,n)=(\theta_{1}(m),\theta_{2}(n)),$$
which signifies that $(G_1\oplus G_2,[.,.]_{G_1\oplus G_2},\Gamma)$ is also a Hom-Lie superalgebra.

\section{isoclinism}
Now onwards we will use $(G,\theta)$ to symbolize Hom-Lie superalgebra. The following generalizations are applicable only in the case of regular Hom-Lie superalgebras by Lemma \ref{l3}. 
\par Now we define isoclinism for Hom-Lie superalgebras.
\begin{definition}\label{d7}
	Consider two regular Hom-Lie superalgebras $(G_1,\theta_1)$ and $(G_2,\theta_2)$. Suppose $\mu :\frac{G_1}{Z(G_1)}\rightarrow \frac{G_2}{Z(G_2)}$ and $\nu :G_1^{'} \rightarrow G_2^{'}$ are two Hom-Lie superalgebra homomorphisms such that:
	\begin{center}
		\begin{tikzpicture}[>=latex]
			\node (x) at (0,0) {\(\frac{G_1}{Z(G_1)}\times \frac{G_1}{Z(G_1)} \)};
			\node (z) at (0,-2) {\(\frac{G_2}{Z(G_2)}\times \frac{G_2}{Z(G_2)}\)};
			\node (y) at (3,0) {\(G_1'\)};
			\node (w) at (3,-2) {\(G_2',\)};
			\draw[->] (x) -- (y) node[midway,above] {$\sigma$};
			\draw[->] (x) -- (z) node[midway,left] {$\mu ^{2} $};
			\draw[->] (z) -- (w) node[midway,below] {$\rho$};
			\draw[->] (y) -- (w) node[midway,right] {$\nu$};
		\end{tikzpicture}\\
	\end{center} 
	\noindent commutes, where $\sigma( \overline{m_{1}},\overline{m_{2}}):=[\overline{m_{1}},\overline{m_{2}}]$ and $\rho( \overline{n_{1}},\overline{n_{2}}):= [\overline{n_{1}},\overline{n_{2}}]$ for $ m_1,m_2 \in G_1 $ and $n_{1},n_2 \in G_2$. Then the pair $(\mu, \nu)$ is said to be {\it homoclinism} and if they are both isomorphisms, then $(\mu, \nu)$ is called {\it isoclinism}.
\end{definition}
\noindent $G_1$ and $G_2$ are said to be isoclinic if $(\mu,\nu)$ is an isoclinism between $G_1$ and $G_2$, which is denoted by $G_1\sim G_2$. One can see that the above notion forms an equivalence relation.

\begin{lemma}\label{l8}
	If $(G_1,\theta_1)$ is a regular Hom-Lie superalgebra and $(G_2,\theta_2)$ is an abelian Hom-Lie superalgebra, then $G_1\sim G_1\oplus G_2$.
\end{lemma}
\begin{proof}
	Since $G_2$ is abelian, we have $Z(G_1\oplus G_2)=Z(G_1)\oplus G_2$. Define the map $$\mu :\frac{G_1}{Z(G_1)}\rightarrow \frac{G_1\oplus G_2}{Z(G_1)\oplus G_2},$$ by $$m+Z(G_1)\rightarrow (m,0)+(Z(G_1)\oplus G_2), $$ for $m\in G_1$. It is easy to verify that the map is well defined. Take two homogeneous elements $n_1,n_2\in G_1$, we have $\mu([n_1+Z(G_1),n_2+Z(G_1)])=[\mu(n_1+Z(G_1)),\mu(n_2+Z(G_1))]$ holds, implies that $\mu$ is a homomorphism. Now we have to show that $\mu \tilde{\theta}_1=\tilde{\phi}\mu$, i.e., 
	$$\mu \tilde{\theta_1}(\overline{n})=\mu(\theta(n)+Z(G_1))=(\theta(n),0)+(Z(G_1)\oplus W)=\tilde{\phi}(n,0)+(Z(G_1)\oplus G_2)=\tilde{\phi}\mu(\overline{n}).$$
	Clearly $\mu$ is a bijection, thus an isomorphism. Next we consider the identity map $\nu : G_1^{'} \rightarrow (G_1\oplus G_2)^{'} = G_1^{'}$ then we get the commutative diagram:
	\begin{center}
		\begin{tikzpicture}[>=latex]
			\node (x) at (0,0) {\(\frac{G_1}{Z(G_1)}\times \frac{G_1}{Z(G_1)}\)};
			\node (z) at (0,-3) {\(\frac{G_1\oplus G_2}{Z(G_1)\oplus G_2}\times \frac{G_1\oplus G_2}{Z(G_1)\oplus G_2}\)};
			\node (y) at (4,0) {\(G_1^{'}\)};
			\node (w) at (4,-3) {\(G_1 ^{'}.\)};
			\draw[->] (x) -- (y) node[midway,above] {$\sigma$};
			\draw[->] (x) -- (z) node[midway,left] {$\mu^2$};
			\draw[->] (z) -- (w) node[midway,below] {$\rho$};
			\draw[->] (y) -- (w) node[midway,right] {$\nu$};
		\end{tikzpicture}
	\end{center}
\end{proof}

To proof the below lemmas one may refer \cite{moneyhun1994,nayak2018}. 
\begin{lemma}\label{l9}
	Suppose that $(G,\theta)$ is a regular Hom-Lie superalgebra and $K$ is a Hom-ideal. Then $G/K \sim G/{(K\cap G^{'})}$. In particular, if $K\cap G^{'}=0$ then $G\sim G/K$. Conversely if $G^{'}$ is finite dimensional and $G\sim G/K$ then $K\cap G^{'}=0$.
\end{lemma}

\begin{corollary}\label{c10}
	Consider two regular Hom-Lie superalgebras $(G_1,\theta_1)$ and $(G_2,\theta_2)$. If $f:G_1\rightarrow G_2$ is an onto homomorphism such that $Ker(f)\cap G^{'}=0$, then $f$ induces an isoclinism between $G_1$ and $G_2$.  
\end{corollary}
\begin{lemma}\label{l11}
	Suppose $\mathcal{C}$ is an isoclinic family of regular Hom-Lie superalgebas. Then
	\begin{enumerate}
		\item $\mathcal{C}$ contains a stem Hom-Lie superalgebra.
		\item Each finite dimensional Hom-Lie superalgebra $T\in \mathcal{C}$ is stem if and only if $T$ has minimal dimension in $\mathcal{C}$.
	\end{enumerate}
\end{lemma}

\begin{lemma}\label{l12}
	Suppose $(\mu,\nu)$ is an isoclinism of Hom-Lie superalgebras $G_1$ and $G_2$. Then the following statements hold: 
	\begin{enumerate}
		\item $\mu(n+Z(G_1))=\nu(n)+Z(G_2)$,
		\item $\nu([n_1,n_2])=[\nu(n_1),n_3]$ for $n_1\in G_1^{'},~n_2\in G_1$, and $n_3+Z(G_2)=\mu(n_2+Z(G_1
		))$.
	\end{enumerate}
\end{lemma}
\begin{lemma}\label{l13}
	Let $(G_1,\theta)$ be a regular Hom-Lie superalgebra and $G_1= G_2 \oplus Z(G_1)$ then $\theta(G_2)\subseteq G_2$. 
\end{lemma}
\begin{proof}
	Suppose $\theta(w) \in Z(G_2)$ for some $0\neq w \in G_2$. For $y \in G_1$ there is a homogeneous element $x \in G_1$ such that $\theta(x)=y$. Thus $0=[\theta(w),y]=[\theta(w),\theta(x)]=\theta([w,x]).$ Since $\theta$ is injective, $[w,x]=0$ for $x \in G_1$. Thus $u \in Z(G_1)$, which is not true. Therefore $\theta(w)\in G_2$.
\end{proof}

\section{Factor set in regular Hom-Lie superalgebras} 
From Lemma \ref{l3}, it can be concluded that if $(G,\theta)$ is regular then $Z(G)$ is a Hom-ideal. For this reason we study factor set only for regular Hom-Lie superalgebras.
\begin{definition}\label{d14}
	A finite dimensional Hom-Lie superalgebra is a triple $(G, [.,.],{\theta})$ over a field $\mathbb{F}$ equipped with a bilinear map;
	$$r:G/Z(G)\times G/Z(G)\rightarrow Z(G),$$ is said to be a factor set if the following properties hold:
	\begin{enumerate}
		\item $r(\overline {n_1} ,\overline {n_2})=-(-1)^{|\overline {n_1}||\overline {n_2}|} r(\overline{n_2},\overline{n_1}),$
		\item $r([\overline{n_1},\overline{n_2}],\tilde {\theta}({\overline{n_3}}))=r(\tilde{\theta}(\overline{n_1}),[\overline{n_2},\overline{n_3}])-(-1)^{|\overline{n_1}||\overline{n_2}|}r(\tilde{\theta}(\overline{n_2}),[\overline{n_1},\overline{n_3}]),$
	\end{enumerate}
	\noindent for homogeneous elements $\overline{n_1},\overline{n_2},\overline{n_3} \in G/Z(G)$ and $\tilde{\theta}$ is a homomorphism $\tilde{\theta} : G/Z(G)\rightarrow G/Z(G) $ satisfying $\tilde{\theta}(\overline{n})=\theta(n)+Z(G)$. The factor set $r$ is said to be  multiplicative if $$r(\tilde{\theta}(\overline{n_1}),\tilde{\theta}(\overline{n_2}))=\theta r(\overline{n_1},\overline{n_2}),$$ for $\overline{n_1},\overline{n_2}\in G/Z(G)$.
\end{definition}
Next Lemma generate a new Hom-Lie superalgebra from a given regular Hom-Lie superalgebra and a factor set on it.
\begin{lemma}\label{l15}
	Suppose $(G,{\theta})$ is a Hom-Lie superalgebra and $r$ is a factor set on it. Set $R=(Z(G),G/Z(G),r)=\{(g,\overline{n}):g\in Z(G),~\overline{n}\in G/Z(G)\}$.
	\item (1) $(R,\phi)$ is a Hom-Lie superalgebra under the component-wise addition  $$[(g_1,\overline{n_1}),(g_2,\overline{n_2})]:=(r(\overline{n_1},\overline{n_2}),[\overline{n_1},\overline{n_2}]),$$
	and the linear map $\phi:R\rightarrow R$ is given by $$\phi(g,\overline{n})=(\theta(g),\tilde{\theta}(\overline{n})).$$
	\item (2) If the factor set $r$ is multiplicative then $(R,\phi)$ is regular.
	\item (3) $Z_{R}=\{(g,0) \in R:x \in Z(G)\} \cong Z(G).$
\end{lemma}
\begin{proof}
	The map $[.,.]$ is linear. Let $(g_1,\overline{n_1})$ and $(g_2,\overline{n_2})$ be two homogeneous elements in $R$, then $|g_1|=|\overline{n_1}|=|(g_1,\overline{n_1})|$ and $|g_2|=|\overline{n_2}|=|(g_2,\overline{n_2})|$. To check the graded skew-symmetric property, consider
	\begin{equation*}
		\begin{split}
			[(g_1,\overline{n_1}),(g_2,\overline{n_2})]&=-(-1)^{|\overline{n_1}||\overline{n_2}|}(r(\overline{n_2},\overline{n_1}),[\overline{n_2},\overline{n_1}])\\
			&=-(-1)^{|\overline{n_1}||\overline{n_2}|}[(g_2,\overline{n_2}),(g_1,\overline{n_1})]\\
			&=-(-1)^{|(g_1,\overline{n_1})||(g_2,\overline{n_2})|}[(g_2,\overline{n_2}),(g_1,\overline{n_1})].
		\end{split}	
	\end{equation*}
	To check the graded Hom-Jacobi identity, take
	\begin{align*}
		&(-1)^{|\overline{n_1}||\overline{n_3}|} [[(g_1,\overline{n_1}),(g_2,\overline{n_2})],\theta(g_3,\overline{n_3})]+(-1)^{|\overline{n_2}||\overline{n_1}|}[[(g_2,\overline{n_2}),(g_3,\overline{n_3})],\theta(g_1,\overline{n_1})]\\
		&\hspace{1cm}+(-1)^{|\overline{n_3}||\overline{n_2}|}[[(g_3,\overline{n_3}),(g_1,\overline{n_1})],\theta(g_2,\overline{n_2})]\\
		&=(-1)^{|\overline{n_1}||\overline{n_3}|}[(r(\overline{n_1},\overline{n_2}),[\overline{n_1},\overline{n_2}]),(\theta(g_3),\tilde{\theta}(\overline{n_3}))]+(-1)^{|\overline{n_2}||\overline{n_1}|}[(r(\overline{n_2},\overline{n_3}),[\overline{n_2},\overline{n_3}]),(\theta(g_1),\tilde{\theta}(\overline{n_1}))]\\
		&\hspace{1cm}+(-1)^{|\overline{n_3}||\overline{n_2}|}[(r(\overline{n_3},\overline{n_1}),[\overline{n_3},\overline{n_1}]),(\theta(g_2),\tilde{\theta}(\overline{n_2}))]\\
		&=(-1)^{|\overline{n_1}||\overline{n_3}|}((r(\overline{n_1},\overline{n_2}),\tilde{\theta}(\overline{n_3})),[[\overline{n_1},\overline{n_2}],\tilde{\theta}(\overline{n_3})])+(-1)^{|n_2||n_1|}((r(\overline{n_2},\overline{n_3}),\tilde{\theta}(\overline{n_1})),[[\overline{n_2},\overline{n_3}],\tilde{\theta}(\overline{n_1})])\\
		&\hspace{1cm}+(-1)^{|n_3||n_2|}((r(\overline{n_3},\overline{n_1}),\tilde{\theta}(\overline{n_2})),[[\overline{n_3},\overline{n_1}],\tilde{\theta}(\overline{n_2})])\\
		&=\big((-1)^{|\overline{n_1}||\overline{n_3}|}(r(\overline{n_1},\overline{n_2}),\tilde{\theta}(\overline{n_3}))+(-1)^{|\overline{n_1}||\overline{n_2}|}(r(\overline{n_2},\overline{n_3}),\tilde{\theta}(\overline{n_1}))+(-1)^{|\overline{n_2}||\overline{n_3}|}(r(\overline{n_3},\overline{n_1}),\tilde{\theta}(\overline{n_2})),\\&\hspace{1cm}(-1)^{|\overline{n_1}||\overline{n_3}|}[[\overline{n_1},\overline{n_2}],\tilde{\theta}(\overline{n_3})]+(-1)^{|\overline{n_2}||\overline{n_1}|}[[\overline{n_2},\overline{n_3}],\tilde{\theta}(\overline{n_1})]+(-1)^{|\overline{n_3}||\overline{n_2}|}[[\overline{n_3},\overline{n_1}],\tilde{\theta}(\overline{n_2})]\big)\\
		&=(0,\overline{0}),
	\end{align*}
	which implies that $(R,\phi)$ is a Hom-Lie superalgebra. Now suppose that $r$ is multiplicative, then $$\phi([(g_1,\overline{n_1}),(g_2,\overline{n_2})])=\phi(r(\overline{n_1},\overline{n_2}),[\overline{n_1},\overline{n_2}])=(\theta r(\overline{n_1},\overline{n_2}),\tilde{\theta}[\overline{n_1},\overline{n_2}]),$$
	and $$[\phi(g_1,\overline{n_1}),\phi(g_2,\overline{n_2})]=\phi[(\theta(g_1),\tilde{\theta}(\overline{n_1})),(\theta(g_2),\tilde{\theta}(\overline{n_2}))]=(r(\tilde{\theta}(n_1),\tilde{\theta}(n_2)),[\tilde{\theta}(n_1),\tilde{\theta}(n_2)]).$$
	As $(G/Z(G),\tilde{\theta})$ is multiplicative, hence $(R,\phi)$ is regular.
	The proof of (3) is obvious.
\end{proof}

Next Lemma indicates that every regular Hom-Lie superalgebra has a factor set.
\begin{lemma}\label{l16}
	For any regular Hom-Lie superalgebra $(G, {\theta})$, there is a factor set $r$ in such a way that $G\cong (Z(G),G/Z(G),r)$.
\end{lemma}
\begin{proof}
	Let us consider a vector superspace $W$ which is the complement of $Z(G)$, i.e., $G=W\oplus Z(G)$. Consider a map $\Psi:G/Z(G)\rightarrow G$ by $\Psi(\overline{n})=\Psi(n+Z(G))=\Psi(w+z+Z(G))=w$ for $\overline{n}\in G/Z(G)$, $w\in W$, and $z\in Z(G)$. Clearly $\Psi$ is a well-defined homogeneous even linear map. We have $\overline{\Psi(\overline{n})}=\overline{n}$. Now for $\overline{n_1}=w_1+z_1$ and $\overline{n_2}=w_2+z_2$, consider $[\overline{n_1},\overline{n_2}]=[w,w']+Z(G)$. Then
	\begin{align*}
		[\Psi(\overline{n_1}),\Psi(\overline{n_2})]-\Psi([\overline{n_1},\overline{n_2}])+Z(G)&=[w,w']-\Psi(\overline{[w,w']})+Z(G)\\
		&=[w,w']- \overline{\Psi(\overline{[w,w']})}=0+Z(G).
	\end{align*}
	So $[\Psi(\overline{n_1}),\Psi(\overline{n_2})]-\Psi([\overline{n_1},\overline{n_2}])\in Z(G)$. Define $$r:G/Z(G)\times G/Z(G)\rightarrow G/Z(G),$$ by $r(\overline{m},\overline{n})=[g(\overline{m}),g(\overline{n})]-g([\overline{m},\overline{n}])$. We have to prove that $r$ is a factor set.\\
	First we have to verify its graded skew-symmetric property,
	\begin{align*}
		r(\overline{m},\overline{n})&=[\Psi(\overline{m}),\Psi(\overline{n})]-\Psi([\overline{m},\overline{n}])\\
		&=-(-1)^{|\overline{m}||\overline{n}|}([\Psi(\overline{n}),\Psi(\overline{m})]-\Psi([\overline{n},\overline{m}]))\\
		&=-(-1)^{|\overline{m}||\overline{n}|}r(\overline{n},\overline{m}),
	\end{align*}
	for $\overline{m},\overline{n}\in G/Z(G)$. Next we have to show graded Hom-Jacobi identity. Before that we have to check $\Psi\tilde{\theta}(\overline{n})=\theta \Psi(\overline{n})$ for $\overline{n}\in G/Z(G)$.\\
	From Lemma \ref{l13}, $\theta(W)\subseteq W$, we have 
	$$\Psi\tilde{\theta}(\overline{n})=\Psi\tilde {\theta} (w+z+Z(G))=\Psi(\theta(w)+Z(G))=\theta(w),$$ and 
	$$\theta \Psi(\overline{n})=\theta(\Psi(w+z+Z(G)))=\theta(w).$$
	\noindent Take,
	\begin{align*}
		r([\overline{n_1},\overline{n_2}],\tilde{\theta} (\overline{n_3}))&=[\Psi([\overline{n_1},\overline{n_2}]),\Psi \tilde{\theta}(\overline{n_3})]-\Psi ([[\overline{n_1},\overline{n_2}],\tilde{\theta}(\overline{n_3})])\\
		&=[[\Psi (\overline{n_1}),\Psi(\overline{n_2})],\theta \Psi(\overline{n_3})]-\Psi([[\overline{n_1},\overline{n_2}],\tilde{\theta}(\overline{n_3})])\\
		&=[\theta \Psi(\overline{n_1}),[\Psi(\overline{n_2}),\Psi(\overline{n_3})]]-(-1)^{|\overline{n_1}||\overline{n_2}|}[\theta \Psi(\overline{n_2}),[\Psi(\overline{n_1}),\Psi(\overline{n_3})]]\\&\hspace{1cm}-\Psi([\tilde{\theta}(\overline{n_1}),[\overline{n_2},\overline{n_3}]])+(-1)^{|\overline{n_1}||\overline{n_2}|}\Psi([\tilde{\theta}(\overline{n_2}),[\overline{n_1},\overline{n_3}]])\\
		&=[\theta \Psi (\overline{n_1}),[\Psi(\overline{n_2}),\Psi(\overline{n_3})]]-\Psi([\tilde{\theta}(\overline{n_1}),[\overline{n_2},\overline{n_3}]])\\&\hspace{1cm}-(-1)^{|\overline{n_1}||\overline{n_2}|}\{[\theta \Psi(\overline{n_2}),[\Psi(\overline{n_1}),\Psi(\overline{n_3})]]-\Psi([\tilde{\theta}(\overline{n_2}),[\overline{n_1},\overline{n_3}]])\}\\
		&=r(\tilde{\theta}(\overline{n_1}),[\overline{n_2},\overline{n_3}])-(-1)^{|\overline{n_1}||\overline{n_2}|}r(\tilde{\theta}(\overline{n_2}),[\overline{n_1},\overline{n_3}]),
	\end{align*}
	for $\overline{n_1},\overline{n_2},\overline{n_3}\in G/Z(G)$. Let us define $\pi :R\rightarrow G$ by $\pi (g,\overline{n})=g+\Psi(\overline{n})$  for $g\in Z(G)$ and $\overline{n}\in G/Z(G)$ where $R=(Z(G),G/Z(G),r)$. Clearly $\pi$ is a well-defined bijective homogeneous even linear map and $\pi([(g_1,\overline{n_1}),(g_2,\overline{n_2})])=[\pi(g_1,\overline{n_1}),\pi(g_2,\overline{n_2})]$.
	In addition $$\theta \pi(g,\overline{n})=\theta(g+\Psi(\overline{n}))=\theta(g+w),$$ and $$\pi \theta(g,\overline{n})=\pi (\theta(g),\tilde{\theta}(\overline{n}))=\theta(g)+\Psi(\theta(w)+Z(G))=\theta(x)+\theta(w),$$ where $n=w+z,~w\in W,$ and $z\in Z(G)$.
	Hence $\pi$ is an isomorphism.\\
\end{proof}

The following Lemma shows that there is a relationship between two stem Hom-Lie superalgebras.
\begin{lemma}\label{l17}
	Let $(G_1,\theta_{1})$ be a stem Hom-Lie superalgebra in an isoclinism family of Hom-Lie superalgebras $\mathcal{C}$. Then for any stem Hom-Lie superalgebra $(G_2,\theta_{2})$ of $\mathcal{C}$, there exists a factor set $r$ over $(G_1,\theta_{1})$ such that $G_2 \cong (Z(G_1),G_1/Z(G_1),r)$. 
\end{lemma}
\begin{proof}
	Let $(\mu,\nu)$ be an isoclinism of regular Hom-Lie superalgebras $(G_1,\theta_{1})$ and $(G_2,\theta_{2})$ then by Lemma \ref{l12} $\pi(Z(G_1))=Z(G_2)$. From Lemma \ref{l16}, there exists a factor set $s$ such that $G_2\cong (Z(G_2),G_2/Z(G_2),s)$. Let us define
	$r:G_1/Z(G_1)\times G_1/Z(G_1) \longrightarrow Z(G_1) $ which is given by
	$r(\overline{n_1},\overline{n_2})=\pi^{-1}(s(\mu(\overline{n_1}),\mu(\overline{n_2})))$ for $\overline{n_1},~\overline{n_2} \in G_1/Z(G_1),$
	where $r$ is a skew-bilinear map. Since $\mu$ is an isomorphism, so $\mu \tilde{\theta_1} = \tilde{\theta_2} \mu $. Thus
	\begin{equation*} 
		\begin{split}
			& r  ([\overline{n_1},\overline{n_2}],\tilde{\theta_1}(\overline{n_3}))\\
			&= \pi^{-1}(s(\mu([\overline{n_1},\overline{n_2}]),\mu \tilde{\theta_1}(\overline{n_3})))\\
			&= \pi^{-1}(s(([\mu(\overline{n_1}),\mu(\overline{n_2})]), \tilde{\theta_2}\mu(\overline{n_3})))\\
			&= \pi^{-1}(s((\tilde{\theta_2}\mu(\overline{n_1}),([\mu(\overline{n_2}),\mu(\overline{n_3})]) ))-(-1)^{|\overline{n_1}||\overline{n_2}|}\pi^{-1}(s(\tilde{\theta_2}\mu(\overline{n_2}),([\mu(\overline{n_3}),\mu(\overline{n_1})])))\\
			&=r(\tilde{\theta_1}(\overline{n_1}),[\overline{n_2},\overline{n_3}])-(-1)^{|\overline{n_1}||\overline{n_2}|} r(\tilde{\theta_1}(\overline{n_2}),[\overline{n_3},\overline{n_1}]).
		\end{split}
	\end{equation*}

	\noindent Therefore $(Z(G_1),G_1/Z(G_1),r)$ is a factor set. Let us denote $R=(Z(G_1),G_1/Z(G_1),r)$ and $S=(Z(G_2),G_2/Z(G_2),s)$. Then by Lemma \ref{l15}, $(R,\phi_{1})$ and $(S,\phi_{2})$ are Hom-Lie superalgebras. Let us define the map
	$\beta:(Z(G_1),G_1/Z(G_1),r)\longrightarrow (Z(G_2),G_2/Z(G_2),s)$  by
	\[ \beta(g,\overline{n})=(\nu(g),\mu(\overline{n})),\]
	which is undoubtedly a bijective even linear map with $\beta([(g_1,\overline{n_1}),(g_2,\overline{n_2})]) 
	=[\beta(g_1,\overline{n_1}),\beta(g_2,\overline{n_2})]$. In addition
	\[\beta \phi_{1} (g,\overline{n})=\beta(\theta_{1}(g),\tilde{\theta_1}(\overline{n}))=(\nu\theta_{1}(g),\mu\tilde{\theta_1}(\overline{n})),\]
	and 
	\[ \phi_{2} \beta (g,\overline{n})=\phi_{2}(\nu(g),\mu(\overline{n}))= (\theta_{2}\nu(g),\tilde{\theta_2}\mu(\overline{n})) .\]
	\noindent Since $\mu$ and $\nu$ are isomorphisms, and $\phi_{1}\beta =\beta \phi_{2}$,
	as a result, $\beta$ is an isomorphism and $G_2\cong (Z(G_1),G_1/Z(G_1),r).$ 
\end{proof}

\begin{lemma}\label{l18}
	Suppose that $(G,\theta)$ is a Hom-Lie superalgebra and $r$,$s$ are two factor sets over $(G,\theta)$. Assume that
	\[R=(Z(G),G/Z(G),r),~~~~~~~~Z_{R}=\{(g,0) \in R:g \in Z(G)\} \cong Z(G),\] \[S=(Z(G),G/Z(G),s),~~~~~~~~Z_{S}=\{(g,0) \in S:g \in Z(G)\} \cong Z(G).\]
	Let $\beta$ be an isomorphism from $R$ to $S$ satisfying $\beta(Z_{R})=Z_{S}$, then the restriction of $\beta$ on $G/Z(G)$ and $Z(G)$ define the automorphisms $\varphi \in Aut(G/Z(G))$ and $\psi \in Aut(Z(G))$, respectively.
\end{lemma}
\begin{proof}
	By Lemma \ref{l15}, $(R,\phi)$ and $(S,\phi)$ are regular Hom-Lie superalgebras. By assumption $\beta(Z_R)=Z_S$, so let us define the quotient Hom-Lie superalgebra $\overline{\beta}:\big(\frac{R}{Z_R},\phi_1\big)\rightarrow \big(\frac{S}{Z_S},\phi_2\big)$  by $\overline{\beta}((g,\overline{n})+Z_R)=\beta(g,\overline{n})+Z_S$ is an isomorphism, where $\phi_1:\frac{R}{Z_R}\rightarrow \frac{R}{Z_R}$ and $\phi_2:\frac{S}{Z_S}\rightarrow \frac{S}{Z_S}$ are even linear maps defined as $\phi_1((g,\overline{n})+Z_R)=\phi_1(g,\overline{n})+Z_R$ and $\phi_2((g,\overline{n})+Z_S)=\phi_2(g,\overline{n})+Z_S$, respectively. Take $\varphi$ such that the following diagram is commutative:
	\begin{center}
		\begin{tikzpicture}[>=latex]
			\node (x) at (0,0) {\(\frac{G}{Z(G)} \)};
			\node (z) at (0,-2) {\(\frac{R}{Z_R}\)};
			\node (y) at (2,0) {\(\frac{G}{Z(G)}\)};
			\node (w) at (2,-2) {\(\frac{S}{Z_S},\)};
			\draw[->] (x) -- (y) node[midway,above] {$\varphi$};
			\draw[->] (x) -- (z) node[midway,left] {$\eta_{1}$};
			\draw[->] (z) -- (w) node[midway,below] {$\overline{\beta}$};
			\draw[->] (y) -- (w) node[midway,right] {$\eta_{2}$};
		\end{tikzpicture}\\
	\end{center}
	where $\eta_{1}$ and $\eta_{2}$ are projection maps, i.e., $\eta_{1}(\overline{n})=(0,\overline{n})+Z_R$ and $\eta_{2}(\overline{n})=(0,\overline{n})+Z_S$. So  $\beta(0,\overline{n})+Z_S=(0,\varphi(\overline{n}))+Z_S$ for $\overline{n}\in G/Z(G)$. Now $$(0,\varphi \tilde{\alpha}(\overline{n}))+Z_S=\beta(0,\theta(n)+Z(G))+Z_S=\beta \phi(0,\overline{n})+Z_S.$$
	On the other hand $$(0,\tilde{\theta} \varphi(\overline{n}))+Z_S=\phi(0,\varphi(\overline{n}))+Z_S=\phi \beta(0,\overline{n})+\phi(g)+Z_S=\phi \beta(0,\overline{n})+Z_S,$$ where $g\in Z_S$. As $\beta$ is an automorphism, i.e., $\beta \phi=\phi \beta$ and $\eta_2$ is injective, we have
	\begin{align*}
		&(0,\varphi \tilde{\theta}(\overline{n}))+Z_S=(0,\tilde{\theta} \varphi(\overline{n}))+Z_S\\
		&\implies \eta_2(\varphi \tilde{\theta}(\overline{n}))=\eta_2(\tilde{\theta}\varphi (\overline{n}))\\
		&\implies \varphi\tilde{\theta}(\overline{n})=\tilde{\theta}\varphi (\overline{n}).
	\end{align*}
	Since $\varphi$ is bijective and $\varphi([\overline{n_{1}},\overline{n_{2}}])=[\varphi(\overline{n_1}),\varphi(\overline{n_2})]$. Hence $\varphi$ is an automorphism. Consider the map $\tilde{\beta}:Z_R\rightarrow Z_S$ is defined as $\tilde{\beta}(g,0)=\beta(g,0)$ for $g\in Z(G)$, is an isomorphism. Define $\psi$ in such a way that the following diagram is commutative:
	\begin{center}
		\begin{tikzpicture}[>=latex]
			\node (x) at (0,0) {\( Z(G) \)};
			\node (z) at (0,-2) {\(Z_{R}\)};
			\node (y) at (2,0) {\(Z(G)\)};
			\node (w) at (2,-2) {\(Z_{S},\)};
			\draw[->] (x) -- (y) node[midway,above] {$\psi$};
			\draw[->] (x) -- (z) node[midway,left] {$\overline{\eta_1}$};
			\draw[->] (z) -- (w) node[midway,below] {$\tilde{\beta}$};
			\draw[->] (y) -- (w) node[midway,right] {$\overline{\eta_2}$};
		\end{tikzpicture}\\
	\end{center}
	\noindent where $\overline{\eta_1}$ and $\overline{\eta_2}$ are projection maps and $\beta(g,0)=(\psi(g),0)$ for $g\in Z(G)$. It is easily viewed that $\psi$ is an automorphism.
\end{proof}

\begin{lemma}\label{l19}
	Suppose $(G,\theta)$ is a Hom-Lie superalgebra. $(R,\phi), (S,\phi), Z_{R}$, and $Z_{S}$ are defined as in Lemma \ref{l18}.
	\begin{enumerate}
		\item Let $\beta : R \longrightarrow S$ be a Hom-Lie superalgebra isomorphism satisfying $\beta(Z_{R})= Z_{S}$. If $\varphi \in Aut(G/Z(G))$ and $\psi \in Aut(Z(G))$ are automorphisms generated by $\beta$ then there is a homogeneous even linear map, $\epsilon: G/Z(G) \longrightarrow Z(G)$ such that \[\psi(r(\overline{n_1}, \overline{n_2})+ \epsilon[\overline{n_1}, \overline{n_2}])=s(\varphi(\overline{n_1}), \varphi(\overline{n_2})).\]
		\item Suppose $\varphi \in Aut(G/Z(G))$, $\psi \in Aut(Z(G))$, and $\tau: G/Z(G) \longrightarrow Z(G)$ is a homogeneous even linear map satisfying $$\psi(r(\overline{n_1}, \overline{n_2})+ \tau[\overline{n_1}, \overline{n_2}])=s(\varphi(\overline{n_1}), \varphi(\overline{n_2})) ,~ \tau \tilde{\theta}=\theta \tau.$$ Then there is an isomorphism $\beta:R \longrightarrow S$ which is generated by $\varphi$ and $\psi$ satisfying $\beta(Z_{R})=Z_{S}$.
	\end{enumerate}
\end{lemma}
\begin{proof}
	To prove the first part refer [Lemma 3.6, 11]. For the proof of second part, we only need to show the commutative property of $\beta$ where $\beta:R\rightarrow S$ is well defined, bijective, and homogeneous linear map of even degree defined as $$\beta(g,\overline{n})=(\phi(g)+\tau(\overline{n}),\psi(\overline{n})).$$
	Now $$\beta \phi(g,\overline{n})=\beta(\alpha(g),\tilde{\theta}(\overline{n}))=(\psi \theta(g)+\tau \tilde{\theta}(\overline{n}),\varphi\tilde{\theta}(\overline{n})).$$
	On another side $$\phi\beta(g,\overline{n})=\phi(\psi(g)+\tau(\overline{n}))=(\theta\psi(g)+\theta\tau(\overline{n}),\tilde{\theta}\varphi(\overline{n})).$$
	As $\varphi$ and $\psi$ are automorphisms and $\tau  \tilde{\alpha}=\alpha \tau$, so $\beta \phi=\phi \beta$.
\end{proof}

\begin{theorem}\label{t20}
	Suppose $(G_1,\theta_{1})$ and $(G_2,\theta_{2})$ are two finite dimensional stem Hom-Lie superalgebras of same parity. Then $G_1 \sim G_2$ iff $G_1 \cong G_2$.
\end{theorem}
\begin{proof}
	Suppose $G_1$ and $G_2$ are two finite dimensional stem Hom-Lie superalgebras such that $G_1\cong G_2$, then it is obvious that $G_1\sim G_2$. To prove the converse part, suppose that $G_1\sim G_2$,
	then by Lemma \ref{l16} and Lemma \ref{l17} we have, $G_1 \cong (Z(G_1), \frac{G_1}{Z(G_1)},r)=R$ and $G_2 \cong (Z(G_1), \frac{G_1}{Z(G_1)}, s)=S$. Since $(G_1,\theta_{1})$ and $(G_2,\theta_{2})$ are regular, so from Lemma \ref{l6}, $(R,\phi_1)$ and $(S,\phi_2)$ are also regular. Consider $(\mu, \nu)$ as the isoclinism between the regular Hom-Lie superalgebras $(R,\phi_1)$ and $(S,\phi_2)$. Precisely $Z_{R}=Z(R)$ and $Z_{S}=Z(S)$. Let $\varphi \in Aut(G_1/Z(G_1))$ be defined by $\mu((0, \overline{n})+Z_{R})= (0, \varphi(\overline{n}))+Z_{S}$ for $\overline{n} \in G_1/Z(G_1)$. Suppose the below diagram is commutative:
	\begin{center}
		\begin{tikzpicture}[>=latex]
			\node (A_{1}) at (0,0) {\(\frac{G_1}{Z(G_1)}\times \frac{G_1}{Z(G_1)} \)};
			\node (A_{2}) at (4,0) {\(\frac{R}{Z_R}\times  \frac{R}{Z_R}\)};
			\node (A_{3}) at (8,0) {\(R'\)};
			\node (B_{1}) at (0,-2) {\(\frac{G_1}{Z(G_1)}\times \frac{G_1}{Z(G_1)}\)};
			\node (B_{2}) at (4,-2) {\(\frac{S}{Z_S}\times \frac{S}{Z_S} \)};
			\node (B_{3}) at (8,-2) {\(S',\)};
			\draw[->] (A_{1}) -- (A_{2}) node[midway,above] {$\lambda$};
			\draw[->] (A_{2}) -- (A_{3}) node[midway,above] {$\theta$};
			\draw[->] (B_{1}) -- (B_{2}) node[midway,below] {$\zeta$};
			\draw[->] (B_{2}) -- (B_{3}) node[midway,below] {$\xi$};
			\draw[->] (A_{1}) -- (B_{1}) node[midway,right] {$\varphi^2$};
			\draw[->] (A_{2}) -- (B_{2}) node[midway,right] {$\mu^2$};
			\draw[->] (A_{3}) -- (B_{3}) node[midway,right] {$\nu$};
		\end{tikzpicture}
	\end{center}
	in which
	\begin{equation*}
		\begin{split}
			\lambda(\overline{n_1}, \overline{n_2}) = ((0, \overline{n_1})+Z_{R}, (0, \overline{n_2})+Z_{R}),\\
			\zeta(\overline{n_1}, \overline{n_2}) =  ((0, \overline{n_1})+Z_{S}, (0, \overline{n_2})+ Z_{S}),\\
			\xi((g_1, \overline{n_1})+Z_{S}, (g_2,\overline{n_2})+Z_{S})= [(g_1,\overline{n_1}),(g_2, \overline{n_2})]=(s(\overline{n_1}, \overline{n_2}), [\overline{n_1}, \overline{n_2}]),\\
			\theta((g_1, \overline{n_1})+Z_{R}, (g_2,\overline{n_2})+Z_{R})=(r(\overline{n_1}, \overline{n_2}), [\overline{n_1}, \overline{n_2}]).
		\end{split}
	\end{equation*}
	Again suppose that $\psi \in Aut(Z(G))$ be given by $\nu(g,0)=(\psi(g),0)$ for $g \in Z(G)$. For $\overline{n_1}, \overline{n_2} \in G/Z(G)$, consider
	$$
	\beta \theta((0, \overline{n_1})+Z_{R}, (0, \overline{n_2})+Z_{R}) = \beta[(0,\overline{n_1}),(0, \overline{n_2})],
	$$
	and in addition \begin{align*}
		\xi \theta((0, \overline{n_1})+Z_{R}, (0, \overline{n_2})+Z_{R})&=\xi((0, \mu(\overline{n_1}))+Z_{S}, (0, \mu(\overline{n_2}))+Z_{S} )\\
		&= [(0, \varphi(\overline{n_1})), (0, \varphi(\overline{n_2}))]\\
		&= (s(\varphi(\overline{n_1}), \varphi(\overline{n_2})), [\varphi(\overline{n_1}), \varphi(\overline{n_2})]).
	\end{align*}
	Thus, we have $\psi[(0,\overline{n_1}),(0, \overline{n_2})]
	=(s(\varphi(\overline{n_1}), \varphi(\overline{n_2})), [\varphi(\overline{n_1}), \varphi(\overline{n_2})])$.
	Consider the map $\tau: G_1^{'}/Z(G_1) \longrightarrow Z(G_1)$ such that
	$$\psi(0, [\overline{n_1}, \overline{n_2}])=(\tau( [\overline{n_1}, \overline{n_2}]),t),$$
	\noindent where $t \in G_1/Z(G_1)$ and thus we get
	\[\psi(r( \overline{a}, \overline{b})+\tau( [\overline{a}, \overline{b}])=s((\varphi(\overline{a}), \varphi(\overline{b})).\]
	To apply Lemma \ref{l12}, we may continue $\tau$ to $G_1/Z(G_1)$ by defining $0$ on the complement of $G_1^{'}/Z(G_1)$ in $G_1/Z(G_1)$. Now we will prove $\tau \tilde{\kappa}=\kappa \tau.$ We obtain
	\[(\tau \tilde{\theta}[\overline{n_1}, \overline{n_2}],t_{1})=\nu(0,\tilde{\theta}[\overline{n_1}, \overline{n_2}])=\nu\phi(0,[\overline{n_1}, \overline{n_2}]).\] 
	\noindent Further $\tilde{\theta}$ is surjective, i.e., $\tilde{\theta} (t)=t_2$ for $t_2 \in G_1/Z(G_1)$. Therefore
	\[(\theta \tau [\overline{n_1}, \overline{n_2}],t_2)= (\theta \tau [\overline{n_1}, \overline{n_2}],\tilde{\theta}(t))=\phi(\tau [\overline{n_1}, \overline{n_2}],t )=\phi\nu (0,[\overline{n_1}, \overline{n_2}]) .\]
	As a result $\theta \tau(\overline{n})=\tau \tilde{\theta}(\overline{n}) $ for $n \in G_1'$. Suppose $G_1=G_1'\oplus U$, then define $\tau$ to be zero in $U$. Then $\theta\tau (\overline{u})=0$ for $u \in U$. Since $\varphi$ is injective, $\varphi(u) \in U$ for $u \in U$, we have 
	\[ \tau \tilde{\theta}(\overline{u})=\tau (\theta(u)+Z(V))=0.\]
	Thus $\tau \tilde{\psi}=\psi \tau$ and now we apply Lemma \ref{l19} to get our results.
\end{proof}

To prove the following Theorems one can see \cite{nayak2019}.
\begin{theorem}\label{t21}
	If $\mathcal{C}$ is an isoclinism family of finite dimensional regular Hom-Lie superalgebras then any $G\in \mathcal{C}$ can be written as $G=P\oplus Q$ where $P$ is a stem Hom-Lie superalgebra and $Q$ is some finite dimensional abelian Hom-Lie superalgebra.
\end{theorem}

\begin{theorem}\label{t22}
	Suppose $(G_1,\theta_1)$ and $(G_2,\theta_2)$ are two regular Hom-Lie superalgebras with same dimensions and same parity. Then $G_1\sim G_2$ if and only if $G_1\cong G_2$.
\end{theorem}
Below example shows that two isoclinic Hom-Lie superalgebras may not be isomorphic when they have different dimensions.	
\begin{example}
	Consider a $(2|1)$ dimensional Hom-Lie superalgebra $(G,\theta_1)$ with basis $\{a_1,a_2|a_3\}$ and commutator relations are defined by;
	$$[a_1,a_3,a_3]=a_1; ~[a_2,a_3,a_3]=a_2,$$
	and all other commutator relations are zero. Then $G^{'}=<a_1,a_2>$ and $Z(G)=0$ and hence $G/Z(G)\cong G$.\\
	Now consider a $(3|1)$ dimensional Hom-Lie superalgebra $(W,\theta_{2})$ with basis $\{a^{\prime}_1,a^{\prime}_2,a^{\prime}_3|a^{\prime}_4\}$ and commutator relations are defined by;
	$$[a^{\prime}_1,a^{\prime}_4,a^{\prime}_4]=a^{\prime}_1;~[a^{\prime}_2,a^{\prime}_4,a^{\prime}_4]=a^{\prime}_2,$$
	and all other commutators are zero. Then $W^{\prime}=<a^{\prime}_1,a^{\prime}_2>$ and $Z(W)=\{a^{\prime}_3\}$ and hence $W/Z(W)=\{\overline{a^{\prime}_1},\overline{a^{\prime}_2}|\overline{a^{\prime}_4}\}$ where $\overline{a^{\prime}_i}=a_i+Z(W)$ for $i=1,2,4$.\\
	We observe that $G^{'}\sim W^{'}$ and $\frac{G}{Z(G)}\cong \frac{W}{Z(W)}$ from which one can deduce that  $G\sim W$ while $dim(G)\neq dim(W)$, i.e., $G$ and $W$ are not isomorphic.
\end{example}

\section{Conclusion}
In this research work, factor set for Hom-Lie superalgebras is defined by using the concept of isoclinism and the existence of factor set for Hom-Lie superalgebras is shown. We conclude that two finite dimensional Hom-Lie
superalgebras having same dimensions are isoclinic if and only if they are isomorphic which is the main result of this work. Later we give an example to show that two Hom-Lie superalgebras satisfying the above conditions must have the same dimensions.

\end{document}